\documentclass[12pt]{amsart}
\usepackage{amssymb, amsfonts}
\usepackage{hyperref}

\newcommand{\<}{\langle}
\renewcommand{\>}{\rangle}

\renewcommand{\a}{\alpha}
\renewcommand{\b}{\beta}
\renewcommand{\d}{\delta}
\newcommand{\D}{\Delta}

\newcommand{\g}{\gamma}

\newcommand{\s}{\sigma}

\renewcommand{\o}{\omega}
\renewcommand{\O}{\Omega}
\renewcommand{\S}{\Sigma}

\newcommand{\z}{\zeta}

\newcommand{\cA}{{\mathcal A}}
\newcommand{\cB}{{\mathcal B}}

\newcommand{\cD}{{\mathcal D}}

\newcommand{\cH}{{\mathcal H}}
\newcommand{\cI}{{\mathcal I}}
\newcommand{\cK}{{\mathcal K}}
\newcommand{\cL}{{\mathcal L}}

\newcommand{\cY}{{\mathcal Y}}

\newcommand{\tL}{{\tilde L}}

\newtheorem{thm}{Theorem}[section]
\newtheorem{cor}[thm]{Corollary} 
\newtheorem{pro}[thm]{Proposition} 

\theoremstyle{definition}   

\newtheorem{defn}[thm]{Definition} 
\newtheorem{exam}[thm]{Example}

\newcommand{\bZ}{{\mathbb Z}}

\newcommand{\bC}{{\mathbb C}}

\newcommand{\bE}{{\boldsymbol E}}

\begin{document}

\title
[Standard deviation is a strongly Leibniz seminorm]
{Standard deviation is a \\  strongly Leibniz seminorm  
}
\author{Marc A. Rieffel}
\address{Department of Mathematics \\
University of California \\ Berkeley, CA 94720-3840}
\email{rieffel@math.berkeley.edu}
\thanks{The research reported here was
supported in part by National Science Foundation grant DMS-1066368.}


\begin{abstract}
We show that standard deviation $\s$ satisfies the Leibniz inequality
$\s(fg) \leq  \s(f)\|g\|  +  \|f\|\s(g)$ for bounded functions $f, g$
on a probability space, where the norm is the supremum norm.
A related inequality that we refer to as ``strong" is also shown 
to hold. We show that these in fact hold also for non-commutative
probability spaces. We extend this to the case of matricial seminorms
on a unital C*-algebra, which leads us to treat also the case of
a conditional expectation from a unital C*-algebra onto a
unital C*-subalgebra.
\end{abstract}

\maketitle
\allowdisplaybreaks

\section*{Introduction}

A seminorm $L$ on a unital normed algebra $\cA$ is said to be
\emph{Leibniz} if $L(1_\cA) = 0$ and
\[
L(AB) \ \leq \ L(A)\|B\| \ + \ \|A\|L(B)
\]
for all $A, B \in \cA$. It is said to be \emph{strongly Leibniz}
if further, whenever $A$ is invertible in $\cA$ then
\[
L(A^{-1}) \ \leq \ \|A^{-1}\|^2L(A).
\]
The latter condition has received almost no attention in the
literature, but it plays a crucial role in \cite{R17}, where I
relate vector bundles over compact metric spaces to
Gromov-Hausdorff distance. See for example the proofs
of propositions 2.3, 3.1, and 3.4 of \cite{R17}.

The prototype
for strongly Leibniz seminorms comes from metric spaces. For
simplicity of exposition we restrict attention here to compact
metric spaces. So let $(X, d)$ be a compact metric space, and let
$C(X)$ be the algebra of continuous complex-valued functions on X, equipped
with the supremum norm $\|\cdot\|_\infty$. For each $f \in C(X)$ let
$L(f)$ be its Lipschitz constant, defined by
\[
(0.1)
 \quad L(f) \ = \ \sup\{|f(x) - f(y)|/d(x,y) : x, y \in X \ \  \mathrm{and} \ \  x \neq y \}   . 
\]
It can easily happen that $L(f) = +\infty$, but the set, $\cA$, of functions
$f$ such that $L(f) < \infty$ forms a dense unital $*$-subalgebra of $C(X)$.
Thus $\cA$ is a unital normed algebra, and $L$ gives a (finite-valued)
seminorm on it that is easily seen to be strongly Leibniz. Furthermore,
it is not hard to show \cite{R5} that the metric $d$ can be recovered from
$L$. Thus, having $L$ is equivalent to having $d$.

My interest in this comes from the fact that this formulation suggests how
to define ``non-commutative metric spaces". Given a non-commutative
normed algebra $\cA$, one can \emph{define} a ``non-commutative metric"
on it to be a strongly Leibniz seminorm on $\cA$. There are then important
and interesting analytic considerations \cite{R5}, but we can ignore them for the
purposes of the present paper.

For my study of non-commutative metric spaces I have felt a need for more
examples and counter-examples that can clarify the variety of phenomena
that can occur. While calculating with a simple class of examples 
(discussed in Section \ref{simple}) I unexpectedly found that I
was looking at some standard deviations. I pursued this aspect, and this
paper records what I found.

To begin with, in Section \ref{deviat} we will see that if $(X, \mu)$ is an 
ordinary probabity measure space and if $A = \cL^\infty(X, \mu)$ is the
normed algebra of (equivalence classes) of bounded measurable
functions on $X$, and if $\s$ denotes the usual standard deviation of 
functions, defined by
\[
\s(f) \ = \ \|f - \mu(f)\|_2 \ =  \ \big(\int |f(x) \ - \ (\int f \ d\mu)|^2 \ d\mu(x)\big)^{1/2}  ,
\]
then $\s$ is a strongly Leibniz seminorm on $\cA$. I would be surprised 
if this fact does not appear somewhere in the vast literature on
probability theory, but so far I have not been able to find it. However,
we will also show that this fact is true for standard deviation
defined for non-commutative probability spaces, such as
matrix algebras equipped with a specified state, and for
corresponding infinite-dimensional algebras (C*-algebras)
equipped with a specified state.

In \cite{R17} essential use is made of ``matricial seminorms"
that are strongly Leibniz. By a matricial seminorm on a 
C*-algebra $\cA$ we mean a family $\{L_n\}$ where $L_n$
is a seminorm on the matrix algebra $M_n(\cA)$ over
$\cA$ for each natural number $n$, and the family
is coherent is a natural way. I very much want to extend the
results of \cite{R17} to the non-commutative setting so that I can
use them to relate ``vector bundles" (i.e. projective modules) over
non-commutative algebras such as those studied in \cite{R7, R21}
that are close for quantum Gromov-Hausdorff distance.
For this reason, in Section \ref{matricial} we begin exploring standard
deviation in this matricial setting. In doing this we find that we need
to understand a generalization of standard deviation to the setting
of conditional expectations from a C*-algebra $\cA$ onto a
sub-C*-algebra $\cB$. That is the subject of Section \ref{condexp}.
It leads to the first examples that I know of for Leibniz seminorms
that are not strongly Leibniz. That is the subject of Section \ref{nonstr}.

We will state many of our results for general unital C*-algebras. But
most of our results are already fully interesting for finite-dimensional
C*-algebras, that is, unital $*$-subalgebras of matrix algebras,
equipped with the operator norm. Thus, readers who are not so familiar 
with general C*-algebras will lose little if in reading this paper
they assume that all of the algebras, and Hilbert spaces, are
finite-dimensional.

Very recently I have noticed connections between the topic of this
paper and the topic of resistance networks. I plan to explore this
connection further and to report on what I find.


\section{Sources of strongly Leibniz seminorms}
\label{sources}

Up to now the only source that I know of for strongly Leibniz 
seminorms consists of ``normed first-order differential calculi".
For a unital algebra $\cA$, a first-order differential calculus
is \cite{GVF} a bimodule $\O$ over $\cA$ together with a derivation
$\d$ from $\cA$ to $\O$, where the derivation (or Leibniz) property is
\[
\d(AB) \ = \ \d(A)B \ + \ A\d(B)
\]
for all $A, B \in \cA$. When $\cA$ is a normed algebra, we can ask
that $\O$ also be a normed bimodule, so that
\[
\|A\o B\|_\O \ \leq \ \|A\| \|\o\|_\O \|B\|
\]
for all $\o$ in $\O$ and all $A, B \in \cA$.
In this case if we set
\[
L(A) \ = \ \|\d(A)\|_\O
\]
for all $A \in \cA$, we see immediately that $L$ is a Leibniz seminorm
on $\cA$. But if $A$ is invertible in $\cA$, then the derivation property
of $\d$ implies that
\[
\d(A^{-1}) \ = \ -A^{-1} \d(A) A^{-1}  .
\]
From this it follow that $L$ is strongly Leibniz. For later use we record
this as:
\begin{pro}
\label{deriv}
Let $\cA$ be a unital normed algebra and let $(\O, \d)$ be a normed first-order
differential calculus over $\cA$. Set $L(A) \ = \ \|\d(A)\|_\O$ for all $A \in \cA$.
Then $L$ is a strongly Leibniz seminorm on $\cA$.
\end{pro}

Many of the first-order differential calculi that occur in practice are ``inner", 
meaning that there is a distinguished element, $\o_0$, in $\O$ such that
$\d$ is defined by
\[
\d(A) \ = \ \o_0 A \ - \ A \o_0   .
\]
Among the normed first-order differential calculi, the ones with the richest
structure are the ``spectral triples'' that were introduced by Alain Connes
\cite{Cn2, Cn7, CMr} in order to define ``non-commutative Riemannian
manifolds" and related structures. In this case $\cA$ should be a
$*$-algebra. Then a spectral triple for $\cA$ consists of a Hilbert space
$\cH$, a $*$-representation $\pi$ (or ``action'' on $\cH$ if the notation 
does not include $\pi$)  of $\cA$ into the algebra $\cL(\cH)$ of
bounded operators on $\cH$, and a self-adjoint operator $D$ on $\cH$
that is often referred to as the ``Dirac" operator for the spectral triple.
Usually $D$ is an unbounded operator, and the requirement is that for
each $A \in \cA$ the commutator $[D, \pi(A)]$ should be a bounded 
operator. (There are further analytical requirements, but we will not
need them here.) By means of $\pi$ we can view $\cL(\cH)$ as a bimodule
$\O$ over $\cA$. Then if we set
\[
\d(A) \ = \ [D, \pi(A)] \ = \ D\pi(A) - \pi(A)D
\]
for all $A \in \cA$, we obtain a derivation from $\cA$ into $\O$. It is
natural to equip $\cL(\cH)$ with its operator norm. If $\cA$ is equipped
with a $*$-norm such that $\pi$ does not increase norms, then $(\O, \d)$
is clearly a normed first-order differential calculus, and we obtain a
strongly Leibniz $*$-seminorm $L$ on $\cA$ by setting
\[
L(A) \ = \ \|[D, \pi(A)]\|   .
\]
We see that $(\O, \d)$ is almost inner, with $D$ serving as the distinguished
element $\o_0$,  the only obstacle being that
$D$ may be unbounded, and so not in $\O$. 

Part of the richness of spectral triples is that they readily provide matricial
seminorms, in contrast to more general normed first-order differential
calculi. This will be fundamental to our discussion in Section \ref{matricial}.

More information about the above sources of strongly Leibniz seminorms
can be found in section 2 of \cite{R21}. One can make some trivial
modifications of the structures described above, but it would be interesting
to have other sources of strongly Leibniz seminorms that are genuinely different.


\section{A class of simple examples}
\label{simple}

In section 7 of \cite{R5} I considered the following very simple spectral
triple. Let $X  =  \{1,2,3\}$, let $\cK  =  \ell^2(X)$, and let $\cB  =  \ell^\infty(X)$
with its evident action on $\cK$ by pointwise multiplication. Let $D$ be the
``Dirac" operator on $\cK$ whose matrix for the standard basis of $\cK$
is
\[
\begin{pmatrix}  0 & 0 & \a_1 \\
                          0 & 0 & \a_2  \\
                         -\a_1 & -\a_2 & 0   
\end{pmatrix}   .
\]
(For ease of bookkeeping we prefer to take our ``Dirac" operators, here and
later, to be skew-adjoint so that the corresponding derivation preserves
$*$. This does not change the corresponding seminorm.)
Then it is easily calculated that for $f \in \cB$, with $f = (f_1, f_2, f_3)$, we have
\[
L(f) \ = \ \big((f_1 - f_3)^2|\a_1|^2 \ + \ (f_2 - f_3)^2|\a_2|^2\big)^{1/2}   .
\]
This is quite different from the usual Leibniz seminorm as defined in
equation (0.1) -- it looks more like a Hilbert-space norm. This example
was shown in \cite{R5} to have some interesting properties. 

This example can be naturally generalized to the case in which
$X \ = \ \{1,\cdots, n\}$ and we have a vector of constants
$\xi \ = \ (\a_1, \cdots , \a_{n-1}, 0)$. To avoid trivial complications we will
assume that $\a_j \neq 0$ for all $j$. For ease of bookkeeping we will
also assume that $\|\xi\|_2 = 1$. It is clear that the last element, $n$,
of $X$ is playing a special role. Accordingly, we set 
$Y = \{1, \cdots , n-1\}$, and we set $\cA = \ell^\infty(Y)$, so that
$\cB = \cA \oplus \bC$. Let $e_n$ denote the last standard basis
vector for $\cK$. Thus $\xi$ and $e_n$ are orthogonal unit
vectors in $\cK$. Then it is easily seen that the evident generalization
of the above Dirac operator $D$ can be expressed as:
\[
D \ = \ \langle \xi, e_n\rangle_c \ - \ \langle e_n, \xi \rangle_c   , 
\]
where for any $\xi, \eta \in \cK$ the symbol $\langle \xi, \eta\rangle_c$
denotes the rank-one operator on $\cK$ defined by
\[
\langle \xi, \eta\rangle_c(\z) \ = \ \xi \langle \eta, \z \rangle_\cK
\]
for all $\z \in \cK$. (We take the inner product on $\cK$ to be linear in
the second variable, and so $\langle \cdot , \cdot \rangle_c$ is linear in the
first variable.)

Our specific unit vector $\xi$ determines a state, $\mu$, on $\cA$
by $\mu(A) = \langle \xi, A\xi \rangle$, faithful because of our assumption
that $\a_j \neq 0$ for all $j$. Then we see that we can generalize to the
situation in which $\cA$ is a non-commutative unital C*-algebra and
$\mu$ is a faithful state on $\cA$ (i.e. a positive linear functional on $\cA$
such that $\mu(1_\cA) = 1$, and $\mu(A^*A) = 0$ implies
$A = 0$). Let $\cH = \cL^2(\cA, \mu)$ be the corresponding GNS
Hilbert space \cite{KR1, Blk2, Str} obtained by completing $\cA$
for the inner product $\langle A, B\rangle_\mu = \mu(A^*B)$, with
its left action of $\cA$ on $\cH$ and its cyclic vector $\xi = 1_\cA$.

Let $\cB = \cA \oplus \bC$ as C*-algebra, and let $\cK = \cH\oplus\bC$
with the evident inner product. We use the evident action of $\cB$
on the Hilbert space $\cK$. Let $\eta$ be $1 \in \bC \subset \cK$, so
that $\xi$ and $\eta$ are orthogonal unit vectors in $\cK$. We then
define a Dirac operator on $\cK$, in generalization of our earlier $D$, by
\[
D \ = \ \langle \xi, \eta\rangle_c \ - \ \langle \eta, \xi \rangle_c   .
\]

We now find a convenient formula for the corresponding strongly
Leibniz seminorm. We write out the calculation in full so as to make
clear our conventions. For $(A, \a) \in \cB$ we have
\begin{eqnarray*}
[D, (A,\a)] &=& (\langle \xi, \eta\rangle_c - \langle \eta, \xi\rangle_c)(A, \a)
 \ - (A,\a)(\langle \xi, \eta\rangle_c - \langle \eta, \xi\rangle_c)  \\
&=& \ \langle \xi, \bar\a\eta\rangle_c - \langle \eta, A^*\xi\rangle_c
- \langle A\xi, \eta\rangle_c + \langle \a\eta, \xi\rangle_c  \\
&=& \ -\langle (A-\a 1_\cA)\xi, \eta\rangle_c - 
\langle \eta, (A^*-\bar\a 1_\cA)\xi\rangle_c   .
\end{eqnarray*}
(From now on we will often write just $\a$ instead of $\a 1_\cA$.)
Because $\eta$ is orthogonal to $B\xi$ for all $B \in \cA$, we see that
the two main terms above have orthogonal ranges, as do their adjoints,
and so
\[
\|[D, (A, \a)\| \ = \ \|(A - \a)\xi\| \vee \|(A^* - \bar \a)\xi\|   ,
\]
where $\vee$ denotes the maximum of the quantities.
But $\xi$ determines the state $\mu$, and so for any $C \in \cA$ we have
$\|C\xi\| = (\mu(C^*C))^{1/2}$. But this is just the norm of $C$ in
$\cL^2(\cA, \mu) = \cH$, which we will denote by $\|C\|_\mu$.
We have thus obtained:

\begin{thm}
\label{thmd}
 Let $\cA$ be a unital C*-algebra, let $\mu$ be a faithful state
 on $\cA$, and let $\cH = \cL^2(\cA, \mu)$, with its action of
 $\cA$ and its cyclic vector $\xi$. Let $\cK$ be the Hilbert
 space $\cH \oplus \bC$, and let 
 $\cB$ be the C*-algebra
 $\cA \oplus \bC$, with its evident representation on $\cK$.
 Let $\eta = 1 \in \bC \subset \cK$. Define a Dirac operator
 on $\cK$ by
 \[
D \ = \ \langle \xi, \eta\rangle_c \ - \ \langle \eta, \xi \rangle_c   
\]
as above. Then for any $(A, \a) \in \cB$ we have
 \[
 L((\cA, \a)) \ = \  \|[D, (A, \a)]\| \ 
 = \ \|(A - \a)\|_\mu \vee \|(A^* - \bar \a\|_\mu   .
 \]
 \end{thm}          
 
 Of course $L$ is
 a $*$-seminorm which is strongly Leibniz. 


\section{Standard deviation}
\label{deviat}

There seems to exist almost no literature concerning quotients
of Leibniz seminorms, but such literature as does exist \cite{BlC, R21}
recognizes that quotients of Leibniz seminorms may well not be
Leibniz. But no specific examples of this seem to be given in the
literature, and I do not know of a specific example, though I
imagine that such examples would not be very hard to find.

For the class of examples discussed in the previous section
there is an evident quotient seminorm to consider, coming from
the quotient of $\cB$ by its ideal $\bC$. This quotient algebra
can clearly be identified with $\cA$. For $L$ as in Theorem
\ref{thmd}  let us denote its quotient by $\tL$, so that
\[
\tL(A) \ = \ \inf\{L((A, \a)): \a \in \bC\}
\]
for all $A \in \cA$. From the expression for $L$ given in
Theorem \ref{thmd} we see that
\[
\tL(A) \ = \ \inf\{ \|(A - \a)\|_\mu \vee \|(A^* - \bar \a\|_\mu: \a \in \bC\}  .
\]
But $\|\cdot\|_\mu$ is the Hilbert space norm on $\cH$, and
$ \|(A - \a)\|_\mu$ is the distance from $A$ to an element in
the one-dimensional subspace $\bC = \bC 1_\cA$ of $\cH$. The
closest element to $A$ in this subspace is just the projection of $A$
into this subspace, which is $\mu(A)$. Furthermore, 
$\mu(A^*) = \overline{\mu(A)}$, and so by taking $\a = \mu(A)$
we obtain:
\begin{pro}
\label{proinf}
The quotient, $\tL$, of $L$ on $\cA$ is given by
\[
\tL(A) \ = \  \|A - \mu(A)\|_\mu \vee \|A^* - \mu(A^*)\|_\mu   .
\]
If $A^* = A$ then $\tL(A) \ = \  \|A - \mu(A)\|_\mu$.
\end{pro}

But for $A^* = A$ the term $\|A - \mu(A)\|_\mu$ is exactly the standard
deviation of $A$ for the state $\mu$, as used in quantum mechanics,
for example on page 56 of \cite{Sdb}. When one expands the inner 
product used to define this term, one quickly obtains, by a well-known
calculation,
\setcounter{equation}{1}
\begin{equation}
\label{eqstd}
\|A - \mu(A)\|_\mu \ = \ (\mu(A^*A) - |\mu(A)|^2)^{1/2}   ,
\end{equation}
which is frequently more useful for calculations of the standard
deviation. I have not seen the standard deviation defined for
non-self-adjoint operators, but in view of all of the above, 
it seems reasonable to
define it as follows:
\setcounter{thm}{2}

\begin{defn}
\label{defstd}
Let $\cA$ be a unital C*-algebra and let $\mu$ be a state
on $\cA$.  We define the \emph{standard deviation} with
respect to $\mu$, denoted by $\s^\mu$, by
\[
\s^\mu(A) =  \|A - \mu(A)\|_\mu \vee \|A^* - \mu(A^*)\|_\mu  .
\]
for all $A\in \cA$.
\end{defn}

As is natural here, we have not required $\mu$ to be faithful.
For simplicity of exposition we will nevertheless continue to
require that $\mu$ be faithful as we proceed. But with some
further arguing this requirement can be dropped.

When I noticed this connection with standard deviation, I said to myself
that surely the standard deviation fails the Leibniz inequality, thus
giving an example of a Leibniz seminorm $L$ that has a quotient
seminorm $\tL$ that is not Leibniz. This expectation was reinforced when 
I asked several probabilists if they had ever heard of the standard
deviation satisfying the Leibniz inequality, and they replied that they 
had not. But when I tried to find specific functions for which the
Leibniz inequality failed, I failed. Eventually I found the following
simple but not obvious proof that the Leibniz inequality does hold.
The proof depends on using the original form of the definition of standard
deviation rather than the often more convenient form given 
in equation \ref{eqstd}.
Define a seminorm, $L_0$, on $\cA$ by
\[
L_o(A) = \|A - \mu(A)\|_\mu  .
\]
We begin with:

\begin{pro}
\label{proleib}
Let notation be as above. Then $L_0$
is a Leibniz seminorm on $\cA$.
\end{pro}

\begin{proof}
Let $A, B \in \cA$. Since $\mu(AB)$ is the closest point in $\bC$ to
$AB$ for the Hilbert-space norm of $\cH$, we will have
\begin{eqnarray*}
L_0(AB) &=& \|AB - \mu(AB)\|_\mu \ \leq \ \|AB-\mu(A)\mu(B)\|_\mu   \\
&\leq& \|AB -A\mu(B)\|_\mu \ + \ \|A\mu(B) - \mu(A)\mu(B)\|_\mu   \\
&\leq& \|A\|_\cA \|B-\mu(B)\|_\mu \ + \ \|A-\mu(A)\|_\mu|\mu(B)|,
\end{eqnarray*}
and since $|\mu(B)| \leq \|B\|_\cA$, we obtain the Leibniz inequality.
\end{proof}

Note that $L_0$ need not be a $*$-seminorm. Because the
maximum of two Leibniz seminorms is again a Leibniz seminorm
according to proposition 1.2iii of \cite{R21}, we obtain from the
the definition of $\s^\mu$ given in Definition \ref{defstd} and from the
above proposition:

\begin{thm}
\label{thmleib}
Let notation be as above. The standard deviation seminorm, $\s^\mu$, is
a Leibniz $*$-seminorm.
\end{thm}

This leaves open the question as to whether $L_o$ and $\s^\mu$ are strongly Leibniz.
I was not able to adapt the above techniques to show that they are. But in
conversation with David Aldous about all of this (for ordinary probability
spaces), he showed me the ``independent copies trick" for expressing
the standard deviation. (As a reference for its use he referred me to
the beginning of the proof of proposition 1 of \cite{DSt}. 
I have so far not found this trick
discussed in an expository book or article.) A few hours after that
conversation I realized that this trick fit right into the normed
first-order differential calculus framework described in Section \ref{sources}.
But when adapted to the non-commutative setting it seems to work
only when $\mu$ is a tracial state (in which 
case $\s^\mu = L_0$). The ``trick'' goes as follows.
Let $\O = \cA \otimes \cA$ (with the minimal C*-tensor-product norm 
\cite{Blk2, KR2}),
which is in an evident way an $\cA$-bimodule. 
 Set $\nu = \mu \otimes \mu$, which
is a state on $\O =\cA \otimes \cA$ as C*-algebra. Thus $\nu$ determines an inner
product on $\O$ whose norm makes $\O$ into a normed bimodule (because
$\mu$ is tracial).
Let $\o_0 = 1_\cA \otimes 1_\cA$. Then for $A \in \cA$ we have
\begin{eqnarray*}
&{}& \|\o_0 A - A \o_0\|^2_\nu \ 
= \ \< 1_\cA \otimes A - A\otimes 1_\cA, 1_\cA \otimes A - A\otimes 1_\cA \>_\nu   \\
&=& \mu(A^*A) \ - \ \mu(A)\mu(A^*) \ - \ \mu(A^*)\mu(A) \ + \ \mu(A^*A) \\
&=& 2(\mu(A^*A) - |\mu(A)|^2) \ = \ 2\|A-\mu(A)\|^2_\mu  
.\end{eqnarray*}
From Proposition \ref{deriv} we thus obtain:

\begin{pro}
\label{prost}
Let notation be as above, and assume that $\mu$ is a tracial state. 
Then $L_0$, and so $\s^\mu$,
is a strongly Leibniz seminorm on $\cA$.
\end{pro}

But by a different path we can
obtain the general case for $\s^\mu$ (but not for $L_0$):

\begin{thm}
\label{thmst}
Let notation be as above (without assuming that $\mu$ is tracial). 
The standard deviation seminorm, $\s^\mu$, is
a strongly Leibniz $*$-seminorm.
\end{thm}

\begin{proof}
Let $E$ be the orthogonal projection from $\cH = L^2(\cA, \mu)$ onto
its subspace $\bC1_\cA$. Note that for $A \in \cA \subseteq \cH$ we
have $E(A) = \mu(A)$. We use $E$ as a Dirac operator, and we let
$L^E$ denote the corresponding strongly Leibniz seminorm,
defined by
\[
L^E(A) = \|[E,A]\|  ,
\]
where we use the natural action of $\cA$ on $\cH$, and the norm is
that of $\cL(\cH)$.

Let $\cH_0$ be the kernel of $E$, which is just the closure of
$\{B-\mu(B): B \in \cA\}$, and is the orthogonal complement
of $\bC1_\cA$. Notice that 
\[
[E,A](1_\cA) = \mu(A) - A  ,
\]
while if $B \in \cH_0 \cap \cA$ then
\[
[E,A](B) = \mu(AB)  .
\]
Thus $[E,A]$ takes $\bC1_\cA$ to $\cH_0$ and $\cH_0$ to $\bC1_\cA$.
We also see that the norm of $[E,A]$ restricted to $\bC1_\cA$ is
$\|A-\mu(A)\|_\mu$.

Notice next that $L^E(A) = L^E(A - \mu(A))$, so we only need consider
$A$ such that $\mu(A) = 0$, that is, $A \in \cH_0$. For such an $A$ we see 
from above that the norm of the restriction of $[E,A]$ to $\cH_0$ is
no larger than $\|A^*\|_\mu$. But because $A \in \cH_0$ we have
$[E,A](A^*) = \mu(AA^*) = \|A^*\|_\mu^2$. Thus the norm of the restriction 
of $[E,A]$ to $\cH_0$ is exactly $\|A^*\|_\mu$. Putting this all together,
we find that
\[
\|[E,A]\| = \|A-\mu(A)\|_\mu \vee \|A^* - \mu(A^*)\|_\mu = \s^\mu(A)
\]
for all $A \in \cA$. Then from Proposition \ref{deriv} we see that $\s^\mu$
is strongly Leibniz as desired.
\end{proof}

We remark that for every Leibniz $*$-seminorm its null-space (where it takes
value 0) is a $*$-subalgebra, and that the null-space of $\s^\mu$ is the
subalgebra of $A$'s such that $\mu(A^*A) = \mu(A^*)\mu(A)$. When such
an $A$ is self-adjoint one says that $\mu$ is ``definite'' on $A$  ---  see
exercise 4.6.16 of \cite{KR1}. 

The above theorem leaves open the question as to whether $L_0$ is strongly
Leibniz when $\mu$ is not tracial. 
I have not been able to answer this question. Computer
calculations lead me to suspect that it is strongly Leibniz
when $\cA$ is finite-dimensional.
We will see in section \ref{nonstr} some examples of 
closely related Leibniz
seminorms that fail to be strongly Leibniz.

I had asked Jim Pitman about the ``strongly'' part of the 
strongly Leibniz property for the case of
standard deviation on ordinary probability spaces, and he surmised that
it might  be generalized in the following way, and Steve Evans quickly
produced a proof. For later use we treat the case of complex-valued
functions, with $\s^\mu$ defined as $L_0$.

\begin{pro}
\label{procalc}
Let $(X, \mu)$ be an ordinary probability space, and let $f$ be
a complex-valued function in $\cL^\infty(X, \mu)$. For any complex-valued Lipschitz 
function $F$ defined on a subset of $\bC$ containing the range of $f$ we will have
\[
\s^\mu(F \circ f) \leq Lip(F) \s^\mu(f)
\]
where $Lip(F)$ is the Lipschitz constant of $F$.
\end{pro}

\begin{proof} 
(Evans) By the independent copies trick mentioned before Proposition
\ref{prost} we have
\begin{align*}
(\s^\mu(F\circ f))^2 \ &= \ (1/2)\int |F(f(x))-F(f(y))|^2 \ d\mu(x) \ d\mu(y)  \\
&\leq (1/2)(Lip(F))^2\int |f(x)-f(y)|^2 \ d\mu(x) \ d\mu(y)   \\
&= (Lip(F))^2(\s^\mu(f))^2   .
\end{align*}
\end{proof}

We can use this to obtain the corresponding non-commutative
version:

\begin{thm}
\label{thmcalc}
Let $\cA$ be a unital C*-algebra and let $\mu$ be a state on $\cA$.
Let $A \in \cA$ be normal, that is,  $A^*A = AA^*$. 
Then for any complex-valued Lipschitz
function $F$ defined on the spectrum of $A$ we have
\[
\s^\mu(F(A)) \ \leq \ Lip(F)\s^\mu(A)   ,
\]
where $F(A)$ is defined by the continuous functional calculus for
normal operators, and $Lip(F)$ is the Lipschitz constant of $F$.
\end{thm}

\begin{proof}
Let $\cB$ be the C*-subalgebra of $\cA$ generated by $A$ and $1_\cA$.
Then $\cB$ is commutative because $A$ is normal, and so 
\cite{KR1, Blk2, Str} $\cB$ is isometrically $*$-algebra isomorphic to
$C(\S)$ where $\S$ is the spectrum of $A$ (so $\S$ is a compact 
subset of $\bC$)
and $C(\S)$ is the C*-algebra of continuous complex-valued functions
on $\S$. (This is basicly the spectral theorem for normal operators.) 
Under this isomorphism $A$ corresponds to the function $f(z) = z$
for $z \in \S \subset \bC$. Then $F(A)$ corresponds to the function
$F = F\circ f$ restricted to $\S$. The state $\mu$ restricts to a state on
$C(\S)$, giving a probability measure on $\S$. Then the desired
inequality becomes
\[
\s^\mu(F) \leq Lip(F)\s^\mu(f)
\]
But this follows immediately from Proposition \ref{procalc}
\end{proof}

It would be reasonable to state the content of Proposition
\ref{procalc} and Theorem \ref{thmcalc} as saying that $\s^\mu$
satisfies the ``Markov'' property, in the sense used for example
in discussing Dirichlet forms.

It is easily checked that for a compact metric space $(X,d)$
and with $L$ defined by equation 0.1 one again has
$L(F\circ f) \leq Lip(F)L(f)$ for $f \in C(X)$ and $F$ defined on the 
range of $f$, so that $L$ satisfies the Markov property. 
But it is not clear to me what happens
already for the case of $f$ in the C*-algebra
$C(X, M_n)$ for $n \geq 2$, with $f^* = f$ or $f$ normal, and 
with $F$ defined
on the spectrum of $f$, and with the operator norm of $M_n$
replacing the absolute value in equation 0.1. 
A very special case that is crucial to
\cite{R17} is buried in the proof of proposition 3.3 of \cite{R17} .
It would be very interesting to know what other classes of
strongly Leibniz seminorms satisfy the Markov property 
for the continuous functional
calculus for normal elements in the way given by
Theorem \ref{thmcalc} .

We remark that by considering the function $F(z) = z^{-1}$ 
Theorem \ref{thmcalc} gives an independent proof of the 
``strongly" property of $\s^\mu$ for normal elements of
$\cA$, but not for general elements. Consequently, if $L_0$
fails to be strongly Leibniz it is because the failure is
demonstrated by some non-normal invertible element
of $\cA$.

Let $\cA = M_n$, the algebra of $n \times n$ complex matrices, 
for some $n$, and let $S(A)$ be the state space
of $\cA$, that is, the set of all states on $\cA$. In this setting
Audenaert proved in theorem 9 of \cite{Aud} that for any $A \in \cA$
we have
\[
\max\{\|A-\mu(A)\|_\mu : \mu \in S(A)\}
\ = \ \min\{\|A-\a\| : \a \in \bC\}   .
\]
In \cite{Mln} the left-hand side is called the ``maximal deviation" of $A$.
A slightly simpler proof of Audenaert's theorem is given in 
theorem 3.2 of \cite{BhS2}.  I thank Franz Luef for bringing \cite{BhS2}
to my attention, which led me to \cite{Aud}. We now generalize
Audenaert's theorem to any unital C*-algebra.

\begin{thm}
\label{thmaud}
Let $\cA$ be a unital C*-algebra. For any $A \in \cA$ set 
$\D(A) =  \min\{\|A-\a\| : \a \in \bC\}$.
Then for any $A \in \cA$ we have
\[
\D(A) \ = \ \max\{\|A-\mu(A)\|_\mu : \mu \in S(A)\}  .
\]
\end{thm}

\begin{proof}
For any $\mu \in S(\cA)$
and any $A \in \cA$ we have $\mu(A^*A) \leq \|A\|^2$,  and so
$\mu(A^*A) - |\mu(A)|^2 \leq \|A\|^2$. Consequently 
$\|A - \mu(A)\|_\mu \leq \|A\|$. But the left-hand side takes value
0 on $1_\cA$, and so $\|A - \mu(A)\|_\mu \leq \|A - \a\|$ for
all $\a \in \bC$. Consequently we have 
\[
\sup\{\|A-\mu(A)\|_\mu : \mu \in S(A)\} \leq \D(A)  .
\]
Thus it suffices to show that for any given $A \in \cA$ there exists
a $\mu \in S(A)$ such that $\|A-\mu(A)\|_\mu = \D(A)$.
By the proof of theorem 3.2 of \cite{R24} there is a $*$-representation $\pi$
of $\cA$ on a Hilbert space $\cH$, and two unit-length vectors $\xi$ and
$\eta$ that are orthogonal, such that $\<\eta, \pi(A)\xi\> = \D(A)$.
For notational simplicity we omit $\pi$ in the rest of the proof.
Let $\mu$ be the state of $\cA$ determined by $\xi$, that is,
$\mu(B) = \<\xi, B\xi\>$ for all $B \in \cA$. Decompose $A\xi$ as
\[
A\xi \ = \ \a\xi \ + \b\eta \ + \g\z 
\]
where $\z$ is a unit vector orthogonal to $\xi$ and $\eta$. Note
that $\b = \<\eta, A\xi\> = \D(A)$. Then
\begin{align*}
& \mu(A^*A) - |\mu(A)|^2 \ = \ \<A\xi, A\xi\> - |\<\xi, A\xi\>|^2  \\
&= |\a|^2 + |\b|^2 + |\g|^2 - |\a|^2 \ = \ |\b|^2 + |\g|^2  = (\D(A))^2 + |\g|^2   . 
\end{align*}
Thus $\|A - \mu(A)\|_\mu = \D(A)$ as desired (and $\g = 0$).
\end{proof}

It would be interesting to have a generalization of Theorem \ref{thmaud}
to the setting of a unital C*-algebra and a unital C*-subalgebra (with
the subalgebra replacing $\bC$ above) along
the lines of theorem 3.1 of \cite{R24}, or to the setting of
conditional expectations discussed in Section
\ref{condexp}.

We remark that theorem 3.2 of \cite{R24} asserts that $\D$
(denoted there by $L$) is a strongly Leibniz seminorm.
We have seen (Proposition \ref{proleib}) 
that each seminorm $A \mapsto \|A-\mu(A)\|_\mu$ is 
Leibniz. This is consistent with the fact that the 
supremum of a family of Leibniz seminorms is
again Leibniz (proposition 1.2iii of \cite{R21}).
Note that $\D$ is a $*$-seminorm even though
 $A \mapsto \|A-\mu(A)\|_\mu$ need not be. 
 This is understandable since $|\<\eta, C^*\xi\>| = |\<\xi, C\eta\>|$
 and we can apply the above reasoning with $\mu$ replaced
 by the state determined by $\eta$. 
 
 Anyway, we obtain:
 
\begin{cor}
\label{coraud}
With notation as above, for every $A \in \cA$ we have
\[
\max \{\s^\mu(A): \mu \in S(\cA)\} \ = \ \D(A).
\]
\end{cor}

It is easy to see that the supremum of a family of Markov seminorms
is again Markov. We thus obtain:

\begin{cor}
\label{cormar}
With notation as above, the seminorm $\D$ is Markov.
\end{cor}

 
\section{Matricial seminorms}
\label{matricial}

Let us now go back to the setting of Section \ref{simple}, with 
$\cB = \cA \oplus \bC$ and $\cK = \cH \oplus \bC$, where
$\cH = \cL^2(\cA, \mu)$. As suggested near the end of Section \ref{sources},
the Dirac operator $D$ defined on $\cK$ in Section \ref{simple} will
define a matricial seminorm  $\{L_n\}$ on $\cB$ 
(more precisely an $\cL^\infty$-matricial
seminorm, but we do not need the definition \cite{Pau} of that here). This works
as follows. Each $M_n(\cB)$ has a unique C*-algebra norm coming
from its evident action on $\cK^n$. Then $D$ determines a Dirac
operator $D_n$ on $\cK^n$, namely the $n \times n$ 
matrix with $D$'s on the diagonal and 0's elsewhere. 
Notice that for any $B \in M_n(\cB)$
the effect of taking the commutator with $D_n$ 
is simply to take the commutator with $D$ of each entry of $B$. 
For any $B \in M_n(\cB)$ we then
set $L_n(B) = \|[D_n, B]\|$. Each $L_n$ will be strongly Leibniz.

It is known \cite{Rua, Pau} that if $\cB$ is any C*-algebra with a
($\cL^\infty$-) matricial seminorm $\{\cL_n\}$, and if $\cI$ is a closed
two-sided ideal in
$\cB$, then we obtain a ($\cL^\infty$-) matricial seminorm on $\cB/\cI$
by taking the quotient seminorm of $L_n$ on $M_n(\cB)/M_n(\cI)$
for each $n$. We apply this to the class of examples that we have
been discussing, with $\cI = \bC \subset \cB = \cA \oplus \bC$. We
denote the quotient seminorm of $L_n$ by $\tL_n$. Our main
question now is whether each $\tL_n$ is Leibniz, or even strongly
Leibniz. 

To answer this question we again first need a convenient
expression for the norm of $[D_n, B]$.
From our calculations preceding Theorem \ref{thmd},
for $\{(A_{jk} , \a_{jk})\} \in M_n(\cB)$ its commutator with $D_n$
will have as entries (dropping the initial minus sign) 
\[
(A_{jk}-\a_{jk})\<\xi, \eta\>_c +
\< \eta, \xi\>_c(A_{jk}-\a_{jk})   .
\]
If we let $V$ denote the element of $M_n(\cL(\cK))$ having
$\<\xi, \eta\>_c$ in each diagonal entry and 0's elsewhere, and
if we let $G$ be the matrix $\{A_{jk} - \a_{jk}\}$, viewed as an
operator on $\cK^n$ that takes $\bC^n$ to 0, then the matrix
of commutators can be written as $GV + V^*G$. Now $G$ carries
$\cK^n$ into $\cH^n \subset \cK^n$  and 
so $GV$ carries $\bC^n \subset \cK^n$ into
$\cH^n$ and carries $\cH^n$ to 0. Similarly $V^*G$ carries $\cH^n$ 
into $\bC^n$ and $\bC^n$ to 0. It follows that
\[
\|GV + V^*G\| = \|GV\| \vee \|V^*G\|   .
\]
But $\|V^*G\| = \|G^*V\|$. Thus we basically just need to unwind the 
definitions and obtain a convenient expression for $\|GV\|$.

Now in an evident way $GV$, as an operator from $\bC^n$ 
to $\cH^n$, is given by the matrix $\{\<G_{jk}\xi, \eta\>_c\}$.
But because $\eta = 1 \in \bC \subset \cK$, we see that for 
$\b \in \bC^n$ we have $GV(\b) = \{(\sum_k G_{jk}\b_k)\xi\}$,
an element of $\cH^n$. Then
\[
\|GV(\b)\|^2 \ = \ \sum_j\|(\sum_kG_{jk}\b_k)\xi\|^2  .
\]
But $\cH =\cL^2(\cA, \mu)$ and $\xi = 1_\cA \in \cH$, 
and so for each $j$ we have
\begin{align*}
&\|(\sum_kG_{jk}\b_k)\xi\|^2 \ = \|\sum_k G_{jk}\b_k\|^2_\mu  \\
&=\< \sum_k G_{jk}\b_k, \sum_\ell G_{j\ell}\b_\ell \>_\mu \
 = \sum_{k, \ell}\bar\b_k \mu(G_{jk}^* G_{j, \ell}) \b_\ell    .
\end{align*}
Thus 
\[
\|GV(\b)\|^2 \ = \ \sum_j\<\b, \ \{\mu(G^*_{jk}G_{j\ell})\}\b\> \
= \<\b, \{\mu(G^*G)_{k\ell}\}\b\>   .
\]
From this it is clear that
\[
\|GV\| \ = \ \|\{\mu(G^*G)_{k\ell}\}\|,
\]
where now the norm on the right side 
is that of $M_n$. View $M_n(\cA)$ 
as $M_n \otimes \cA$, and set 
\[
\bE_n^\mu \ = \ id_n \otimes \mu 
\]
where $id_n$ is the identity map of $M_n$ onto itself, so
that $\bE_n^\mu$ is a linear map from $M_n(\cA)$ onto
$M_n$. Then
\[
 \|\{\mu(G^*G)_{k\ell}\}\| \ = \ \|\bE_n^\mu(G^*G)\|   .
\]
For any $H \in M_n(\cA)$ set
\[
\|H\|_\bE \ = \ \|\bE_n^\mu(H^*H)\|^{1/2}   .
\]
The conclusion of the above calculations can then be formulated as:

\begin{pro}
\label{procon}
With notation as above, we have
\[
L_n((A, \a)) = \|A-\a\|_\bE \vee \|A^* - \bar\a\|_\bE
\]
for all $(A,\a) \in M_n(\cB)$.
\end{pro}

Now $\bE_n^\mu$ is an example of a ``conditional expectation'',
as generalized to the non-commutative setting \cite{Blk2, KR2} 
(when we view
$M_n$ as the subalgebra $M_n \otimes 1_\cA$ of $M_n(\cA)$).
Thus to study the quotient, $\tL_n$, of $L_n$ we are led to explore our themes 
in the setting of general conditional expectations.


\section{Conditional expectations}
\label{condexp}

Let $\cA$ be a unital C*-algebra and
let $\cD$ be a unital C*-subalgebra of $\cA$ (so $1_\cA \in \cD$).
We recall \cite{Blk2, KR2} that a \emph{conditional expectation}
from $\cA$ to $\cD$ is a bounded linear projection, $\bE$, from $\cA$ onto
$\cD$ which is positive, and has the property that for $A \in \cA$
and $C, D \in \cD$ we have
\[
\bE(CAD) = C\bE(A)D.
\]
(This latter property is often called the ``conditional expectation
property''.) It is known \cite{Blk2, KR2} that conditional expectations
are of norm 1, and in fact are completely positive. One says
that $\bE$ is ``faithful'' if $\bE(A^*A) = 0$ implies that $A=0$.
For simplicity of exposition we will assume that our conditional
expectations are faithful. Given a conditional expectation $\bE$,
one can define a $\cD$-valued inner product on $\cA$ by
\[
\<A,B\>_\bE = \bE(A^*B)
\]
for all $A,B \in \cA$. (See section 2 of \cite{R26}, and \cite{Blk2}.)
From this we get a corresponding (ordinary) norm on $\cA$, defined
by
\[
\|A\|_\bE \ = \ (\|\bE(A^*A)\|_\cD)^{1/2}   .
\]
Actually, to show that this is a norm one needs a suitable
generalization of the Cauchy-Schwartz inequality, for which
see proposition 2.9 of \cite{R26}, or \cite{Blk2}. From the 
conditional expectation property one sees that for
$A, B \in \cA$ and $D \in \cD$ one has
\[
\<A, BD\>_\bE = \<A, B\>_\bE D   .
\]
Accordingly, one should view $\cA$ as a right $\cD$-module.
Since it is evident that $(\<A, B\>_\bE)^* = \< B, A\>_\bE$,
we also have $\<AD, B\>_\bE = D^*\<A, B\>_\bE  $.
It follows that $\|AD\|_\bE \leq \|A\|_\bE\|D\|_\cD$.
When $\cA$ is completed for the norm $\|\cdot\|_\bE$, the
above operations extend to the completion, and one obtains
what is usually called a right Hilbert $\cD$-module \cite{Blk2}.

In this setting we can imitate much of what we did earlier. 
Accordingly, set $\cB = \cA \oplus \cD$. On $\cB$ we can
define a seminorm $L$ by 
\[
L_0((A, D)) = \|A-D\|_\bE.
\]
(Note that $L_0$ need not be a $*$-seminorm.) To see that
$L_0$ is Leibniz, we should first notice that for any $A, B \in \cA$
since $B^*A^*AB \leq \|A\|^2 B^*B$ and $\bE$ is positive, we
have $\bE(B^*A^*AB) \leq \|A\|^2\bE(B^*B)$, so that
\setcounter{equation}{0}
\begin{equation}
\label{bndop}
\|AB\|_\bE \leq \|A\|_\cA\|B\|_\bE .
\end{equation}
 \setcounter{thm}{1}
We can now check that $L_0$ is Leibniz. For $A, B \in \cA$ 
and $C,D \in \cD$ we have 
\begin{align*}
L_0((A,C)(B,D)) &= \|AB -CD\|_\bE \leq \|AB-AD\|_\bE + \|AD-CD\|_\bE   \\
&\leq \|A\|_\cA\|B-D\|_\bE + \|A-C\|_\bE\|D\|_\cA   \\
&\leq \|(A,C)\|_\cB L_0((B,D)) + L_0((A, C))\|(B, D)\|_\cB,
\end{align*}
as desired. Furthermore, $L_0$ is strongly Leibniz, for if $A^{-1}$
and $D^{-1}$ exist, then
\begin{align*}
L_0((A,D)^{-1}) &= \|A^{-1}-D^{-1}\|_\bE = \|A^{-1}(D-A)D^{-1}\|_\bE  \\
&\leq \|A^{-1}\|_\cA \|A-D\|_\bE \|D^{-1}\|_\cA
\leq \|(A,D)^{-1}\|_\cB^2 L_0((A,D)),
\end{align*}
as desired. Since $L_0$ need not be a $*$-norm, we will also want
to use $L_0(A)\vee L_0(A^*)$. Then it is not difficult to put the above
considerations into the setting of the spectal triples mentioned
in Section \ref{sources}, along the lines developed in Section
\ref{simple}. But we do not need to do this here.

We can now consider the quotient, $\tL_0$, of $L_0$ on the quotient
of $\cB$ by its ideal $\cD$, which we naturally identify with $\cA$,
in generalization of what we did in Section \ref{deviat}. Thus we
set
\[
\tL_0(A) = \inf\{ L_0(A-D): D \in \cD\}   .
\]
But we can argue much as one does for Hilbert spaces to obtain:

\begin{pro}
\label{proquo}
For every $A \in \cA$ we have
\[
\tL_0(A) \ = \|A - \bE(A)\|_\bE   .
\]
\end{pro}

\begin{proof}
Suppose first that $\bE(A) =0$. Then for any $D \in \cD$
\begin{align*}
(L_0(A-D))^2 \ &= \|\bE((A-D)^*(A-D))\|_\cD    \\
&= \ \| \bE(A^*A) - D^*\bE(A) -\bE(A^*)D+D^*D\|_\cD   \\
&=\|\bE(A^*A)+D^*D\|_\cD \ \geq \|\bE(A^*A)\|_\cD   .
\end{align*}
Thus 0 is a (not necessarily unique) closest point in $\cD$
to $A$ for the norm $\|\cdot\|_\bE$. Thus $\tL_0(A) = \|A\|_\bE$.
For general $A$ note that $\bE(A-\bE(A)) = 0$. From the above
considerations it follows that $\bE(A)$ is a closest point in $\cD$
to $A$.
\end{proof}
Note that again this expression for $\tL_0$ need not be a $*$-seminorm.
In view of the discussion in Section \ref{deviat} it is appropriate to make:

\begin{defn}
\label{defsdex}
With notation as above, for $A \in \cA$ set
\[
\s^\bE(A) = \tL_0(A) \vee \tL_0(A^*) = \|A-\bE(A)\|_\bE \vee \|A^*-\bE(A^*)\|_\bE  ,
\]
and call it
the \emph{standard deviation} of $A$ \emph{with respect to} $\bE$.
\end{defn}

We can now argue much as we did in the proof of Proposition \ref{proleib}
to obtain:

\begin{pro}
\label{prostleib}
With notation as above, both $\tL_0$ and $\s^\bE$ are Leibniz seminorms.
\end{pro}

\begin{proof}
Let $A, B \in \cA$. By the calculation in the proof of Proposition \ref{proquo} we know
that $\bE(A)\bE(B)$ is no closer to $AB$ for the norm $\|\cdot\|_\bE$
than is $\bE(AB)$. Thus
\begin{align*}
\tL_0(AB) \ &= \ \|AB-\bE(AB)\|_\bE \ \leq \ \|AB-\bE(A)\bE(B)\|_\bE  \\
& \leq \|A(B-\bE(B))\|_\bE \ + \ \|(A-\bE(A))\bE(B)\|_\bE    \\
& \leq \|A\|_\cA \tL_0(B) \ + \ \tL_0(A)\|B\|_\cA   ,
\end{align*}
where we have used equation 5.2 and, implicitly, the conditional
expectation property.
Thus $\tL_0$ is Leibniz. As mentioned earlier, the maximum of two
Leibniz seminorms is again Leibniz, and so $\s^\bE$ too is Leibniz.
\end{proof}

This leaves open the question as to whether $\tL_0$ and $\s^\bE$ are
strongly Leibniz. We will try to imitate the proof of Theorem \ref{thmst}.
We have mentioned earlier that $\cA$, equipped with its 
$\cD$-valued inner product and completed for the corresponding norm,
is a right Hilbert $\cD$-module. If $Z$ is any right Hilbert $\cD$-module,
the appropriate corresponding linear operators on $Z$ are the bounded
adjointable right $\cD$-module endomorphisms (as in definition 2.3 of \cite{R26},
or in \cite{Blk2}), that is, the norm-bounded endomorphisms $T$ for which there
is another such endomorphism, $T^*$, such that
$\<y, Tz\>_\bE = \<T^*y, z\>_\bE$ for all $y, z \in Z$. (This is not 
automatic.) These endomorphisms form a C*-algebra for the operator
norm.

For our situation of $\cA$ equipped with the $\cD$-valued inner
product given by $\bE$, the operators that we are about to use all
carry $\cA$ into itself, and so we do not need to form the completion,
as long as we check that the operators are norm-bounded and have
adjoints. We will denote the algebra of such operators by
$\cL^\infty(\cA, \bE)$, in generalization of our earlier $\cL^\infty(\cA, \mu)$.
It is a unital pre-C*-algebra. 

Each $A \in \cA$ determines an operator in $\cL^\infty(\cA, \bE)$ via
the left regular representation. We denote this operator by $\hat A$.
The proof that $\hat A$ is norm-bounded is essentially equation
\ref{bndop}. It is easily checked that the adjoint of $\hat A$ is
$(A^*)\hat{}$, and that in this way we obtain a $*$-homomorphism
from $\cA$ into $\cL^\infty(\cA, \bE)$. Because $\bE$ is faithful,
this homomorphism will be injective, and so isometric.

Perhaps more surprising is that $\bE$ too acts as an operator in 
$\cL^\infty(\cA, \bE)$. (See proposition 3.3 of \cite{R26}.)
By definition $\bE$ is a right $\cD$-module endomorphism.
For any $A \in \cA$ we have
\[
\<\bE(A), \bE(A)\>_\bE = \bE(\bE(A^*)\bE(A)) = \bE(A^*)\bE(A).
\]
But $\bE(A^*)\bE(A) \leq \bE(A^*A)$ by the calculation (familiar
for the variance, and related to equation 3.2 above) that
\[
0 \leq \bE((A^* - \bE(A^*))(A-\bE(A))) = \bE(A^*A) - \bE(A^*)\bE(A).
\]
Thus $\|\bE(A)\|_\bE \leq \|A\|_\bE$, so that $\bE$ is a norm-bounded
operator. Furthermore, for $A, B \in \cA$ we have
\begin{align*}
\<A, \bE(B)\>_\bE \ &= \ \bE(A^*\bE(B)) \ = \ \bE(A^*)\bE(B)  \\
&= \ \bE(\bE(A^*)B) \ = \ \<\bE(A), B\>_\bE  ,
\end{align*}
so that $\bE$ is ``self-adjoint''. When we view $\bE$ as an element
of $\cL^\infty(\cA, \bE)$ we will denote it by $\hat \bE$.

Let us now use $\hat \bE$ as a ``Dirac operator'' to obtain a
strongly Leibniz $*$-seminorm, $L^\bE$, on $\cA$. Thus $L^\bE$
is defined by
\[
L^\bE(A) = \|[\hat \bE, \hat A]\|   ,
\]
where the norm here is that of $\cL^\infty(\cA, \bE)$. We now
unwind the definitions to obtain a more convenient expression
for $L^\bE$. Notice that $\hat \bE^2 = \hat \bE$. Now if $\cA$ is
any unital algebra and if $a, e \in \cA$ with $e^2 = e$, then
because $[a, \cdot]$ is a derivation of $\cA$, we find that
$e[a,e]e = 0$. Similarly we see that $(1-e)[a, e](1-e) = 0$.
Let $\cY$ be the kernel of $\hat \bE$, so that it consists of
the elements of $\cA$ of the form $A-\bE(A)$. 
Note that $\cY$ and $\cD$ are ``orthogonal'' for $\<\cdot , \cdot \>_\bE$,
and that
$\cA = \cY \oplus \cD$. The calculations just above
show that $[\hat \bE, \hat A]$ carries $\cD$ into $\cY$
and $\cY$ into $\cD$. From this it follows that
\[
\|[\hat \bE, \hat A]\| \ = \ \|\hat \bE[\hat \bE, \hat A](I - \hat \bE)\| \vee
\|(I-\hat \bE)[\hat \bE, \hat A]\hat \bE\|   
\]
for all $A \in \cA$, where $I$ is the identity operator on $\cA$. But note that 
\[
(|\hat \bE[\hat \bE, \hat A](I - \hat \bE))^* \ 
= \ -(I-\hat \bE)[\hat \bE, \hat A^*]\hat \bE   .
\]
Thus we basically only need a convenient expression 
for $\|(I-\hat \bE)[\hat \bE, \hat A]\hat \bE\|$, and the latter is equal to
$\|[\hat \bE, \hat A]|_\cD\|$.

Now for $D \in \cD$ we have 
\begin{align*}
\|[\hat \bE, \hat A](D)\|_\bE \ &= \ \|\bE(AD) - A\bE(D)\|_\bE \ = \ \|(\bE(A)-A)D\|_\bE  \\
 & \leq \|A-\bE(A)\|_\bE \|D\|_\cA   .
\end{align*}
From this and the result when $D = 1_\cA$ we see that
\[
\|[\hat \bE, \hat A]|_\cD\| \ = \ \|A-\bE(A)\|_\bE = \tL_0(A)   .
\]
It follows that 
\[
L^\bE(A)) \ = \ \|A-\bE(A)\|_\bE \vee \|A^*-\bE(A^*)\|_\bE \ = \ \s^\bE(A)
\]
for all $A \in \cA$.
In view of what was said in Section \ref{sources} about first-order
differential calculi, we have thus obtained:

\begin{thm}
\label{thmexleib}
With notation as above, $\s^\bE$ is a strongly Leibniz $*$-seminorm.
\end{thm}

We can immediately apply this to the matricial setting of Section
\ref{matricial}. For that setting and any $n$ we have $\bE = \bE_n^\mu$.
Then, in the notation of the present setting, the conclusion of
Proposition \ref{procon} is again that 
\[
L_n((A,\a)) = \|A-\a\|_\bE \vee  \|A^*-\bar\a\|_\bE
\]
for all $(A,\a) \in M_n(\cB)$. Note that for the present situation, 
the $L$ of the earlier part of this section is given exactly by
$L_0((A,\a)) = \|A-\a\|_\bE$. Then from Proposition \ref{proquo}
we see that
\[
\tL_n(A) \ = \ \|A-\bE(A)\|_\bE \vee \|A^*-\bE(A^*)\|_\bE
\]
for any $A \in \cA$. And the right-hand side is just the corresponding
standard deviation, which we will denote by $\s_n^\bE$. 
Then from Theorem \ref{thmexleib}
we obtain:

\begin{thm}
Let $\cA$ be a unital C*-algebra and let $\mu$ be a faithful state on $\cA$.
For each natural number $n$ let $\bE_n^\mu$ be the corresponding
conditional expectation from $M_n(\cA)$ onto $M_n \subset M_n(\cA)$,
and let $\|\cdot\|_{\bE_n^\mu}$ be the associated norm. Then the standard
deviation $\s_n^\mu$ on $M_n(\cA)$ defined by
\[
\s_n^\mu(A) \ = \ \|A-\bE_n^\mu(A)\|_{\bE_n^\mu} \vee 
\|A^*-\bE_n^\mu(A^*)\|_{\bE_n^\mu}
\]
for all $A \in M_n(\cA)$ is a strongly Leibniz $*$-seminorm.
The family $\{\s_n^\mu\}$ is a strongly Leibniz ($\cL^\infty$)-matricial
$*$-seminorm on $\cA$.
\end{thm}


\section{Leibniz seminorms that are not strongly Leibniz}
\label{nonstr}

Let us return now to the case of a general conditional expectation 
$\bE:\cA \to \cD$. We saw in Proposition \ref{prostleib} that the
seminorm $\tL_0$ on $\cA$ defined by $\tL_0(A) = \|A-\bE(A)\|_\bE$
is a Leibniz seminorm. So we can ask whether it too is strongly
Leibniz. We will now show that it need not be. One evening while at
a conference I began exploring this question. It occurred
to me to consider what happens to unitary elements of $\cA$. 
If $U$ is a unitary element of $\cA$ and if $\tL_0$ is strongly Leibniz, 
then we will have
\[
\tL_0(U^{-1}) \leq \tL_0(U) \quad \mathrm{and} \quad \tL_0(U) \leq \tL_0(U^{-1})
\]
so that $\tL_0(U^{-1}) = \tL_0(U)$. Since $U^{-1} = U^*$, we would thus
have $\tL_0(U^*) = \tL_0(U)$. If $\tL_0$ is a $*$-seminorm, then this is
automatic. But $\tL_0$ may not be a $*$-seminorm. Now
\begin{align*}
\tL_0(U) \ &= \ \|U-\bE(U)\|_\bE \ = \ \|\bE((U^* - \bE(U^*))(U-\bE(U))\|_\cA^{1/2}   \\
&= \ \|1_\cA - \bE(U^*)\bE(U)\|_\cA^{1/2}   .
\end{align*}
So the question becomes whether $\|1_\cA - \bE(U^*)\bE(U)\|_\cA$ can
be different from $\|1_\cA - \bE(U)\bE(U^*)\|_\cA$. But
$\|1_\cA - \bE(U)\bE(U^*)\|_\cA$ is equal to $1-m$ where $m$ is
the smallest point in the spectrum of $\bE(U)\bE(U^*)$.
Now the spectrum of $\bE(U^*)\bE(U)$ is equal to that of
$\bE(U)\bE(U^*)$ except possibly for the value 0. (See
proposition 3.2.8 of
of \cite{KR1}.) Thus the question becomes: Is there an example of
a conditional expectation $\bE:\cA \to \cD$ and a unitary element
$U$ of $\cA$ such that $\bE(U)\bE(U^*)$ is invertible but
$\bE(U^*)\bE(U)$ is not invertible? The next day I asked this
question of several attendees of the conference who had some
expertise is such matters. The following morning, shortly before I
was to give a talk on the topic of this paper, Sergey Neshveyev 
gave me the following example (which I have very slightly reformulated).

\begin{exam}
\label{exnesh}
Suppose that one can find a unital C*-algebra $\cD$ containing two partial
isometries $S$ and $T$ and two unitary operators $V$ and $W$ such
that, for $R = S +T$, we have
\begin{itemize}
\item[i)] $R^*R$ is invertible but $RR^*$ is not invertible,
\item[ii)] $S^* = VTW$   .
\end{itemize}
Then let $\cA = M_2(\cD)$, and define a unitary operator $U$
in $\cA$ by
\[
U \ = \ 
\begin{pmatrix}  V & 0  \\
                          0 & 1                         
\end{pmatrix}   
\begin{pmatrix}  T & (1-TT^*)^{1/2}  \\
                          -(1-T^*T)^{1/2} & T^*                         
\end{pmatrix} 
\begin{pmatrix}  W & 0  \\
                          0 & 1                         
\end{pmatrix}   .
\]
(See the solution of problem 222 of \cite{Hlm}.)
Let $\tau$ denote the normalized trace, i.e. the tracial state,
on $M_2$, and let $\bE = \tau \otimes id$ where $id$ is the
identity map on $\cA$. Then $\bE$ is a conditional expectation
from $\cA$ onto $\cD$, where $\cD$ is identified with $I_2 \otimes \cD$
in $M_2 \otimes \cD = \cA$. Then
\[
\bE(U) \ = \ (S^* + T^*)/2 \ = \ R^*/2  .
\]
Consequently $\bE(U)\bE(U^*)$ is invertible but
$\bE(U^*)\bE(U)$ is not invertible, as desired.

It remains to show that there exist operators $S, T, V, W$ satisfying
the properties listed above. Let $\cH = \ell^2(\bZ)$ with its
standard orthonormal basis $\{e_n\}$, and let $\cD = \cL(\cH)$.
Let $B$ denote the right bilateral shift operator on $\cH$, so
$Be_n = e_{n+1}$ for all $n$. Let $J$ be the unitary operator
determined by $Je_n = e_{-n}$ for all $n$, and let $P$ be the 
projection determined by $Pe_n = e_n$ if $n \geq 0$ and 0
otherwise. Set $S = JBP$ and $T= BPJ$, and set $R=S+T$.
It is easily checked that
$R^*Re_n = e_n$ if $n \neq 0$ while $R^*Re_0 = 2e_0$,
so that $R^*R$ is invertible, but $R^*e_0 = 0$ so that
$RR^*$ is not invertible, as desired. 
Furthermore, if we set $V = B^{-1}$ and $W = B$, then it is easily
checked that $S^* = VTW$ as desired.
\end{exam}

The above example provides the first Leibniz seminorm $L$ that I
know of that is not strongly Leibniz, and so can not be obtained from 
a normed
first-order differential calculus. But motivated by the above example
we can obtain simpler examples, which are not so closely related
to conditional expectations.

\begin{exam}
\label{excomp}
Let $\cA$ be a unital C*-algebra, and let $P$ be a projection in
$\cA$ (with $P^*=P$). Let $P^\perp = 1_\cA - P$. Define  $\g$
on $\cA$ by
\[
\g(A) = P^\perp AP
\]
for all $A \in \cA$. Then $\g$ is usually not a derivation, but we have
\begin{align*}
\g(AB) &= P^\perp ABP - P^\perp APBP + P^\perp APBP   \\
&= P^\perp A(P^\perp BP) + (P^\perp A P)BP
= \g(A\g(B) + \g(A)B)) 
\end{align*}
for all $A, B \in \cA$. Now set
\[
L(A) = \|\g(A)\|
\]
for all $A \in \cA$. Because $\g$ is norm non-increasing, it is
clear from the above calculation that $L$ is a Leibniz seminorm. 
It is also clear that $L$ may not be a $^*$-seminorm. We remark
that if $L$ is restricted to any unital C*-subalgebra of $\cA$,
without requiring that $P$ be in that subalgebra, we obtain again
a Leibniz seminorm on that subalgebra.

We can ask whether $L$ is strongly Leibniz. The following
example shows that it need not be. 
Much as in Example \ref{exnesh},
we use the fact that if $L$ is strongly Leibniz then for any unitary
element $U$ in $\cA$ we must have $L(U^*) = L(U)$. 

Let $\cH = \ell^2(\bZ)$ with its
standard orthonormal basis $\{e_n\}$, and let $\cA = \cL(\cH)$.
Let $U$ denote the right bilateral shift operator on $\cH$, so
$Ue_n = e_{n+1}$ for all $n$, and let $P$ be the 
projection determined by $Pe_n = e_n$ if $n \geq 0$ and 0
otherwise. Then it is easily seen that $P^\perp UP= 0$ while
$P^\perp U^*Pe_0 = e_{-1}$. Thus $L(U) = 0$ while
$L(U^{-1})=1$. 

We now show that if $P\cA P$ is finite dimensional, or
at least has a finite faithful trace, then
$L(U^*) = L(U)$ for any unitary element $U$ of $\cA$.
Notice that
\[
\|P^\perp UP\|^2 = \|PU^*P^\perp UP\| = \|P - PU^*PUP\|
= 1-m
\]
where $m$ is the minimum of the spectrum of $PU^*PUP$
inside $P\cA P$. On applying this also with $U$ replaced by
$U^*$, we see, much as in Example \ref{exnesh}, that
$L(U) \neq L(U^*)$ exactly if one of  $PU^*PUP$
and $PUPU^*P$ is invertible in $P\cA P$ and the other is not.
This can not happen if $P\cA P$ has a finite faithful trace.
But this does not prove that $L$ is strongly Leibniz in that
case.

For the general case of this example, if we set
\[
L_s(A) = \max\{L(A), L(A^*)\}   ,
\]
then, much as earlier, $L_s$ will be a Leibniz $*$-seminorm.
But in fact, $L_s$ will be strongly Leibniz. This is because
\[
[P, A] = PAP^\perp \ - \ P^\perp AP,
\]
so that 
\[
\|[P,A]\| = \|PAP^\perp\| \vee  \|P^\perp AP\| \ = \ L_s(A).
\]
This is all closely related to the Arveson distance formula \cite{Arv},
as shown to me by Erik Christensen at the time when I developed theorem
3.2 of \cite{R24}.

\end{exam}

But the above examples depend on the
fact that $L$ is not a $*$-seminorm. It would be interesting to have
examples of Leibniz $*$-seminorms that are not strongly Leibniz. 
It would also be interesting
to have examples for which $\cA$ is finite-dimensional. (Note that right after proposition 1.2 of \cite{R21} there is an example of 
a Leibniz $*$-seminorm that is not
strongly Leibniz, but this example depends crucially on the Leibniz
seminorm taking value $+\infty$ on some elements.)


\def\dbar{\leavevmode\hbox to 0pt{\hskip.2ex \accent"16\hss}d}
\providecommand{\bysame}{\leavevmode\hbox to3em{\hrulefill}\thinspace}
\providecommand{\MR}{\relax\ifhmode\unskip\space\fi MR }
\providecommand{\MRhref}[2]{%
  \href{http://www.ams.org/mathscinet-getitem?mr=#1}{#2}
}
\providecommand{\href}[2]{#2}

\end{document}